\documentclass{amsart}

\usepackage{amsmath}
\usepackage{amsfonts}
\usepackage{amssymb}

\usepackage{float}
\usepackage[final]{graphicx}
\usepackage{enumerate}

\usepackage{pifont}
\usepackage[utf8x]{inputenc}

\usepackage[T1]{fontenc}
\usepackage{cancel}

\usepackage{tikz}
\usetikzlibrary{decorations.markings}
\usetikzlibrary{decorations.pathreplacing}
\usetikzlibrary{arrows,shapes,positioning,patterns}
\usetikzlibrary{knots}
\tikzstyle{none}=[inner sep=0pt]
\pgfdeclarelayer{edgelayer}
\pgfdeclarelayer{nodelayer}
\pgfsetlayers{edgelayer,nodelayer,main}
\usepackage{url} 
\usepackage{xcolor}

\newcommand{\bR}{\mathbb R}
\newcommand{\bS}{\mathbb S}
\newcommand{\bN}{\mathbb N}

\newcommand{\bC}{\mathbb C}

\newcommand{\bD}{\mathbb D}

\newcommand{\bZ}{\mathbb Z}

\renewcommand{\epsilon}{\varepsilon}

\newcommand{\x}{\times}

\DeclareMathOperator{\wsp}{sp}
\DeclareMathOperator{\ssp}{\widetilde{sp}}

\newcommand{\lk}{\ell k}

\newtheorem{thm}{Theorem}[section]
\newtheorem{lemma}[thm]{Lemma}
\newtheorem{cor}[thm]{Corollary}
\newtheorem{rem}[thm]{Remark}

\newtheorem{prop}[thm]{Proposition}
\newtheorem{exa}[thm]{Example}

\newtheorem{defn}[thm]{Definition}

\numberwithin{equation}{section}
\newcommand{\dfn}[1]{\emph{#1}}

\begin{document}

\title[Slice-torus invariants, combinatorial invariants and positivity]{Slice-torus link invariants, combinatorial invariants, and positivity conditions.}

\author{Carlo Collari}

\maketitle
\begin{abstract}
We prove some necessary conditions for a link to be either concordant to a quasi-positive link, quasi-positive, positive, or the closure of a positive braid.~%
The main applications of our results are a characterisation of positive links with unlinking number $1$ and $2$, and a combinatorial criterion to test if a positive link is the closure of a positive braid.
Finally, we compile a table of all positive and positive-braid  prime links with less than $8$ crossings.
\end{abstract}

\section{Introduction}

In this note we study slice-torus and combinatorial invariants of multi-component links -- {as opposed to} knots -- in $\bS^3$ satisfying various positivity conditions.~%
The aim of this note is dual;~%
the first goal is to provide a necessary condition for a link to be concordant to a quasi-positive link, and compare it to other known obstructions.~%
The second goal is to investigate certain combinatorial invariants of quasi-positive, positive and positive-braid links, providing obstructions for a link to fall in one of the aforementioned classes.~%
As applications we give a characterisation of positive links with unlinking number $1$ and $2$, and a combinatorial criterion to test if a positive link is the closure of a positive braid. Finally, we compile a table of all positive and positive-braid  prime links with less than $8$ crossings.

In what follows all manifolds and sub-manifolds are smooth and oriented, and boundaries are endowed with the induced orientation. 

\subsection{Concordance to quasi-positive links}

A \dfn{link} is a $1$-dimensional closed sub-manifold of the $3$-sphere $\bS^3$,  and a \dfn{knot} is a connected link. The letter $L$ will be reserved for links, whereas the letter $K$ will be reserved for knots. Two links are \dfn{equivalent} up to orientation-preserving ambient isotopy. As customary in the literature, we shall abuse the notation and make no difference between a link and its equivalence class.

Quasi-positive links are a special family of links which sits in the intersection between knot theory and the theory of complex curves in $\bC^2$. More precisely, quasi-positive link are those links obtained as the transverse intersection of a complex curve in $\bC^2$  with $\bS^3 = \{  (w,z)\in \bC^2\,\vert\,\vert w \vert^2 + \vert z\vert^2 =1\}$ -- see \cite{BoileauOrevkov,RudolphAC}. 
This hints to the fact that quasi-positive links should have special properties with respect to concordance.~%
\begin{defn}
Let $L_0$ and $L_1$ be two $\ell$-component links. Then, $L_0$ and $L_1$ are \dfn{concordant} if there exists properly embedded disjoint annuli $A_1 \sqcup ... \sqcup A_\ell \subset \bS^3\times [0,1]$, such that both $A_i\cap \bS^3\x\{0\}$ and $A_i\cap \bS^3\x\{1\}$ are non-empty for each $i$, and
\[  \left( A_1 \cup ... \cup A_\ell \right) \cap \bS^{3}\x \{ 0 \} = L_0,\quad\text{and}\quad  \left( A_1 \cup ... \cup A_\ell \right) \cap \bS^{3}\x \{ 1 \} = rL_1^*,\]
where $rL$ denotes the link $L$ with its orientation reversed, and $L^*$ denotes the image of $L$ under an orientation reversing diffeomorphism of $\bS^3$.
A link concordant to an unlink is called (strongly) \dfn{slice}.
\end{defn}

If one considers topologically locally flat embeddings instead of smooth embeddings, quasi-positive links do not present any special properties; in fact, every link is topologically locally flat concordant to a (strongly) quasi-positive link \cite{BorodzickFeller}. It follows that all concordance invariants which are also invariant under topological locally flat concordance cannot detect quasi-positive links. In particular, the linking matrix and, more generally, all classical concordance invariants obtained from the Seifert form, such as the Levine-Tristram signatures, fail to distinguish these links -- cf.~\cite{RudolphI,Rudolph05}.

The situation is different if one looks from the smooth perspective. 
First, we recall that slice-torus link invariants are {a special class of} $\bR$-valued concordance invariants of links, which vanish on unlinks {-- see} \cite{CavalloCollari}  {for the precise definition}.~%
Prominent examples of slice-torus link invariants are the Ozsv\'ath-Szab\'o-Rasmussen $\tau$-invariant \cite{CavalloTau,OzsvathSzabo,RasmussenThesis}, and (a re-scaling of) the Rasmussen $s$-invariant \cite{BeliakovaWehrli,Rasmussen}.~%
It is known that the slice-torus link invariants provide an obstruction to being concordant to a quasi-positive link;~%
indeed, it follows from \cite[Theorem 3.2]{CavalloCollari} (see \cite{CavalloBennequin,Lewarksln}, and references therein, for knots, and special slice-torus link invariants) that: if $L$ is an $\ell$-component link concordant to a quasi-positive link, and $\nu$ is a slice-torus invariant, then
\begin{equation}
2\nu(L) -\ell =-\chi_4(L),
\label{eq:OldConditionsQP}
\end{equation}
where $\chi_{4}$ is the (smooth) \dfn{slice Euler characteristic}, that is
\[ \chi_4(L) = \max \{ \chi(\Sigma)\,|\, \Sigma \subset\bD^4\ \text{compact surface with no closed components},\, \partial \Sigma = \Sigma \cap \partial \bD^4= L\}.\]
An immediate consequence of \eqref{eq:OldConditionsQP} is that there are links which are not concordant to any quasi-positive link, eg.~the figure-eight knot $4_1$ is such that $\nu(4_1) = 0$, and $-\chi_4(4_1) =1$ \cite{KnotInfo}. ~%

While there are some results for special classes of links including knots, to the best of the author's knowledge \eqref{eq:OldConditionsQP} is the only known general obstruction to being concordant to a quasi-positive link.~%

\begin{rem}
To be fair, there is another general obstruction to being concordant to a quasi-positive link. This obstruction is stated in terms of the Hom-Wu invariant $\nu^{+}$ \cite{CavalloLocally,HomWu} -- which is not a slice-torus link invariant.~%
Concretely, it follows from  \cite[Proposition~1.6]{CavalloLocally} (cf.~\cite[Proposition 3]{HomWu} for knots) that if $L$ is concordant to a quasi-positive link, then
\[2\nu^+(L) -\ell =-\chi_4(L) \]
However, this obstruction is strictly weaker than Equation \eqref{eq:OldConditionsQP} specialised to the case $\nu = \tau$ \cite{CavalloLocally,HomWu}. 
\end{rem}

In this note we provide a new obstruction to being concordant to a quasi-positive link based on the slice-torus link invariants and linking numbers.~%

\begin{thm}\label{teo:MainFullQP}
Let $L$ be a link, and let $\nu$ be a slice-torus link invariant.~%
If $L$ is concordant to a quasi-positive link, then
\begin{equation}
\nu(L) - \sum_{i=1}^{k} \nu(L_i) \leq \sum_{i<j} \lk (L_i,L_j),
\label{eq:teo_conc_qp}
\end{equation}
for each partition of $L$ into sub-links $L_1,...,L_k$.
\end{thm}

Note that Theorem \ref{teo:MainFullQP} poses constraints on the linking matrix of a quasi-positive link in terms of its sub-links.~%
This is slightly surprising for two reasons;~%
first, as we mentioned above, all linking matrices can be realised by quasi-positive links.~%
Second, all links arise as sub-links of a quasi-positive link by adding a single unknotted component. Since
we were unable to find a proof of this fact, which is well-known two experts, we included one in Appendix \ref{app:sublinks}.~%

Going back to our result, we note that the obstructions provided by Theorem~\ref{teo:MainFullQP} and by~\eqref{eq:OldConditionsQP} are independent. %
First, we give an example where \eqref{eq:OldConditionsQP} provides a stronger obstruction than Theorem~\ref{teo:MainFullQP}.
Recall that $L$ is the \dfn{split union} of $L_1, ..., L_k$, we write  $L = L_1 \sqcup \cdots \sqcup L_k$, if $L= L_1 \cup .... \cup L_k$ and the $L_i$'s are separated by embedded copies of $\bS^2$.
\begin{exa}
Slice-torus link invariants are additive under split union~\cite[Definition~2.1~(B)]{CavalloCollari}. Thus, the split union of a figure-eight knot and an unlink does not satisfy \eqref{eq:OldConditionsQP}, but satisfies \eqref{eq:teo_conc_qp}.
\end{exa}

On the other hand, we use Theorem \ref{teo:MainFullQP} to prove the following result.

\begin{thm}\label{teo:ExampleConctoQP}
Denote by $L(k)$ the two-bridge link in Figure \ref{fig:L_hk}.~For all $k\in \bN$, and any slice-torus link invariant $\nu$, we have
\[ 2\nu(L(k)) - 2 = -\chi_4(L(k)) = 2k-2.\] 
However, $L(k)$ is not concordant to any quasi-positive link for all $k\geq 1$.~%
In particular, the Whitehead link $L(1)$ is not concordant to any quasi-positive link.~%
\end{thm}
\begin{proof}
Using the diagram $D(k)$ in Figure~\ref{fig:L_hk}, and \cite[Theorem~1.3]{CavalloCollari} (cf.~Proposition \ref{pro:combinatorialbound}), one can compute $\nu(L(k)) = k$.~%
Now, a quick inspection of $D(k)$ shows that $L(k)$ can be unlinked with $k$ crossing changes on a single component (cf.~\cite[Proposition 2.7]{CavalloCollariConway}), therefore $-\chi_4(L(k)) \leq 2k-2$. %
Since $2\nu(L) -\ell \leq -\chi_4(L)$, by  \cite[Proposition~2.11]{CavalloCollari}, we obtain the desired equality. %
Finally, $L(k)$ has two trivial components with linking number zero. The result follows directly from Theorem~\ref{teo:MainFullQP}.
\end{proof}
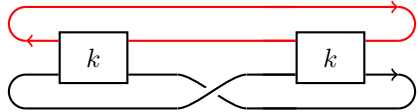
\begin{figure}[ht]

\begin{tikzpicture}[thick, scale = .9]
\draw[,red] (0,.75) arc (90:270:.25);
\draw (0,-.25) arc (90:270:.25);
\draw[<-,red] (0,.25) -- (4,.25);
\draw[] (0,-.25) -- (4,-.25);
\draw[red] (0,.75) -- (4,.75);
\draw[] (0,-.75) -- (4,-.75);

\draw[fill, white] (.5,.375) rectangle (1.5,-.375);
\draw (1.5,.375) rectangle (.5,-.375);
\node at (1,0) {$k$};

\draw[fill, white] (2.25,-.895) rectangle (3.25,-.125);
\draw[] (2.24,-.25) .. controls +(.25,0) and +(-.25,0) .. (3.26,-.75);
\draw[white, line width = 5] (2.24,-.75) .. controls +(.25,0) and +(-.25,0) .. (3.26,-.25);
\draw[] (2.24,-.75) .. controls +(.25,0) and +(-.25,0) .. (3.26,-.25);

\begin{scope}[shift = {+(-2.5,0)}]
\draw[red] (6,.25) -- (8,.25);
\draw[->] (6,-.25) -- (8,-.25);
\draw[->,red] (6,.75) -- (8,.75);
\draw[] (6,-.75) -- (8,-.75);

\draw[fill, white] (6.5,.375) rectangle (7.5,-.375);
\draw (6.5,.375) rectangle (7.5,-.375);
\draw[,red] (8,.75) arc (90:-90:.25);
\draw (8,-.25) arc (90:-90:.25);

\node at (7,0) {$k$};
\end{scope}

\end{tikzpicture}
\caption{The diagram $D(k)$ for the link $L(k)$. The boxes denote $k$ positive full-twists.}
\label{fig:L_hk}
\end{figure}

We conclude the first part of the paper by studying the  concordance self-linking number $sl_c$.~%
We shall be more precise in Section \ref{sec:background}, but we briefly recall the definitions and some properties of $sl_{max}$ and $sl_{c}$.
Each link can be represented non-uniquely as a closed braid~\cite{Alexander}, and to each braid $\beta$ is associated its self-linking number $sl(\beta)\in\bZ$.~%
However, this number depends on the braid used to represent the given link.~%
To obtain a link invariant and a concordance invariant, respectively, one can define
\[ sl_{max}(L) = \max\{ sl(\beta)\,|\,\text{the closure of }\beta\text{ is } L \},\]
and
\[ sl_{c}(L) = \max\{ sl(\beta)\,|\,\text{the closure of }\beta\text{ is concordant to } L \}.\]
Rudolph's \dfn{slice-Bennequin inequality} \cite{RudolphObstructionToSlice} tells us that, for each braid $\beta$ representing a link $L$, we have
\[ sl(\beta) \leq -\chi_4(L).\]
It follows that both $sl_{max}$ and $sl_c$ are finite numbers.
Furthermore, the slice-Bennequin inequality is sharp for quasi-positive braids~\cite[Section~3]{RudolphObstructionToSlice}.
Hence, if $L$ is either a quasi-positive link, or is concordant to a quasi-positive link, then $sl_{max}(L) = -\chi_4(L)$, or $sl_{c}(L)= - \chi_4(L)$, respectively.\\
The slice-Bennequin inequality can be refined to a \dfn{slice-torus Bennequin inequality}; that is, for each $\ell$-component link $L$, and each slice-torus link invariant $\nu$, we have
\[ sl_{max}(L) \leq sl_{c}(L) \leq 2\nu(L) - \ell \leq -\chi_4(L).\]
The first inequality is trivial. The second inequality is well-known to experts, and can be inferred in this general setting from the combinatorial bound \cite[Theorem 1.4]{CavalloCollari} applied to the closure of a braid (see Corollary \ref{cor:slice-torusBennequin}). Finally, the third inequality follows from \cite[Proposition 2.11]{CavalloCollari}.
In particular, the equality $sl_{c}(L) = -\chi_4(L)$ gives a strictly stronger obstruction to being concordant to a quasi-positive link than Equation \eqref{eq:OldConditionsQP} -- cf.~Remark \ref{rem:CvLiu}.

Unfortunately, the concordance self-linking number is incredibly difficult to compute even in simple cases. Which is the reason why, to the best of the author's knowledge, this quantity has never been considered before. In this paper we shall compute $sl_c$ for links which are closures of pure braids (Proposition \ref{prop:slc and pure}). Moreover, we  completely characterise the closures of alternating pure braids which are concordant to a quasi-positive link (Proposition \ref{prop:altpurebraids}).

\begin{rem}\label{rem:CvLiu}
Despite how difficult is to compute $sl_c$, sometimes it can be estimated and used effectively as an obstruction.
For instance, \cite[Proposition 1.6]{CavalloLiu} uses the maximal self-linking number $sl_{max}$ to show that the link $L(k)$ in Figure \ref{fig:L_hk} is not quasi-positive.~%
A similar reasoning, where $sl_{max}$ is replaced by $sl_{c}$, can be used to (re-)prove Theorem~\ref{teo:ExampleConctoQP} -- see~Proposition \ref{prop:slmetc vs lk} and Example~\ref{ex:slcandqp}.
\end{rem}

\begin{rem}\label{rem:MainVSslc}
The proof of Theorem \ref{teo:MainFullQP} uses the equality $sl_{max} = 2\nu - \ell$, rather than quasi-positivity itself. In particular, Theorem \ref{teo:MainFullQP} still holds if we replace the hypothesis "$L$ is concordant to a quasi-positive link" with ``$L$ such that $sl_c(L)= 2\nu(L) -\ell$''. Thus the condition $sl_c(L) = - \chi_4(L)$ provides a stronger, if less computable, obstruction than Theorem~\ref{teo:MainFullQP}.
\end{rem}

\subsection{Obstructions to positivity conditions}

In the second part of this paper we deal with equivalence rather than concordance. Our aim is to provide necessary conditions for a link to be quasi-positive, positive, or the closure of a positive braid, and investigate the consequences of these conditions.~%
Our first result is the following necessary condition for certain links to be quasi-positive.

\begin{thm}\label{teo:obstruction_qp}
Let $L$ be a quasi-positive link. If there exists a partition of $L$ into sub-links $L_1,...,L_k$ such that
\begin{equation}
\label{eq:hp_obs_qp}
 \nu(L) - \sum_{i=1}^{k} \nu(L_i) = \sum_{i<j} \lk (L_i,L_j),
\end{equation}
then
\[ 2\nu(L_i) - \ell_i = sl_{max}(L_i),\quad \text{for all }i\in \{ 1 ,..., k\}.\]
\end{thm}

Exactly as in the case of Theorem \ref{teo:MainFullQP}, see Remark \ref{rem:MainVSslc}, the proof of Theorem \ref{teo:obstruction_qp} uses the condition $sl_{max}= 2\nu -\ell$ rather than quasi-positivity itself.
Therefore, Theorem \ref{teo:obstruction_qp} provides a weaker obstruction than the conditions $sl_{max} = 2\nu -\ell$ and $sl_{max} = - \chi_4$.
{Nonetheless, Theorem~\ref{teo:obstruction_qp} is independent from the condition $2 \nu -\ell = - \chi_4$ (Example \ref{exa:obs-qp}).}

Now, we restrict our attention to positive links, which are special quasi-positive links \cite{Nakamura,RudolphPisSQP}.~%
More precisely, we use combinatorial link invariants to establish some conditions for a link to be positive.~%

Before stating our  result, we need to establish some terminology.~%
Recall that link is \dfn{split} if it is the split union of two links, and is \dfn{completely split} if it is the split union of knots.
The \dfn{splitting number} $\wsp$ of a link $L$ is the minimum number of crossing changes needed to obtain a completely split link from $L$. The splitting number must not be confused with what we call the \dfn{strong splitting number} $\ssp$, which has a similar definition but the only crossing changes allowed are those between distinct components. We remark that sometimes, in the literature, what we call strong splitting number is referred to as splitting number, and what we call splitting number is sometimes called weak splitting number. Finally, \dfn{unlinking number} $u(L)$ is the minimal number of crossing changes to obtain an unlink from $L$. The unlinking number of a knot is also called \dfn{unknotting number}.
Now, we are ready to state our result.

\begin{thm}\label{teo:obstruction_p}
If $L=K_1 \cup \cdots \cup K_\ell$ is a positive link, then we have%
\begin{equation}
\lk (L) = \wsp (L) = \ssp(L) = u(L) - \sum_{i= 1}^{\ell} u(K_i),
\label{eq:obs_pos}
\end{equation}~%
where $\lk(L)$ is \dfn{the total linking number}, that is the sum of $\lk(K_i,K_j)$, for  all $i<j$.
\end{thm}

Before going to the applications of Theorem \ref{teo:obstruction_p}, we pause to comment on our result. %
First, we compare our result with other known criteria: the non-negativity of the linking matrix, the positivity of sub-links, and the negativity of the signature $\sigma$ \cite{Phavenegsigniii,Phavenegsignii,Phavenegsign}.\\
Let us start by proving that Equation \eqref{eq:obs_pos} implies the non-negativity of the linking matrix.  Denote by $\vert \lk \vert (L)$ the sum of all the absolute values of the linking numbers between different components. Lemma \ref{lem:lkleqm} implies that
\[ \lk(L) \leq |\lk| (L) \leq \ssp(L).\]
Therefore, if $\lk(L) =\ssp(L)$ then $\vert \lk \vert(L) =\lk (L)$. Hence all linking numbers must equal their absolute value, and thus they must be non-negative.
Moreover, it is also known that if the link is positive and is not completely split, then there must be at least a non-zero linking number. This is also implied by Equation \eqref{eq:obs_pos}; in fact, it follows directly from the definitions that a link is completely split if, and only if, $\wsp(L) = 0$. Therefore, if a link is positive and is not completely split, then \eqref{eq:obs_pos} implies $\lk(L) = \wsp(L) >0$. So at least one among the linking numbers must be positive.
In the next proposition we show that Theorem~\ref{teo:obstruction_p} provides a strictly stronger obstruction than the linking matrix, and  that it is independent from the positivity of sub-links, and the sign of the signature.
\begin{prop}\label{pro:indep_pos_obs}
There exists a family of $2$-component links $\{ L'(n)\}_{n\geq 2}$ with positive components, negative signature, and positive linking matrix, such  that $\wsp (L'(n)) \neq \ssp(L'(n))$ for all $n$.
Furthermore, there are infinitely many  $2$-component links with positive signature, and a negative component, which satisfy \eqref{eq:obs_pos}.
\end{prop}
Finally, we test the efficiency of Theorem \ref{teo:obstruction_p} on prime links with less than $8$ crossings. The result is contained in Table \ref{tab:small_linksI}. In this table, which is of independent interest, we listed all prime links with less than $8$ crossings, their positivity, their braid-positivity, and whether or not our criteria (ie.~Theorems \ref{teo:obstruction_p} and \ref{teo:obstruction_pb}) can be used to effectively obstruct these properties.

Now, let us turn to the applications of Theorem \ref{teo:obstruction_p}. Nakamura in {\cite[Theorem 5.1]{Nakamura}} proved that if $K$ is a positive knot, then $u(K)=1$ if, and only if, $K$ is a (non-trivial) twist knot.~%
As an application of Theorem \ref{teo:obstruction_p}, we extend Nakamura's result to multi-component links.

\begin{prop}\label{prop:u=1}
Let $L$ be a positive link. Then, $u(L) = 1$ if, and only if, $L$ is the split union of a (possibly empty) unlink with either the positive Hopf link, or a (non-trivial) twist knot.
\end{prop}

Moreover, we use  Theorem \ref{teo:obstruction_p} to obtain a characterisation of all links with unknotting number $2$.

\begin{prop}\label{prop:u=2}
Let $L$ be a positive link. Then, $u(L) = 2$ if, and only if, $L$ is the split union of a (possibly empty) unlink and one of the following
\begin{enumerate}
\item a positive link with total linking number $2$ and unknotted components;
\item  a connected sum of a positive Hopf link and a positive twist knot;
\item the split union of a positive Hopf link and a positive twist knot;
\item  the split union of two non-trivial positive twist knots;
\item  a positive knot with unknotting number $2$.
\end{enumerate} 
Furthermore, each of the above families contains infinitely-many links.
\end{prop}

Finally, we focus on links which are obtained as the closure of positive braids, which we call  \dfn{positive-braid links}.~%
In this case we can strengthen the statement of Theorem \ref{teo:obstruction_p}.

\begin{thm}\label{teo:obstruction_pb}
If $L=K_1 \cup \cdots \cup K_\ell$ is a positive-braid link, then
\[ \nu(L) -\sum_{i=1}^{\ell} \nu(K_i) = \lk  (L) = \wsp(L) =\ssp(L), \]
 for any slice-torus link invariant $\nu$.~%
Furthermore, given any partition of $L$ into sub-links, say $L_1, ..., L_k$, we have
\[ \nu(L) -\sum_{i=1}^{k} \nu(L_i)= u(L) - \sum_{i= 1}^{k} u(L_i) = \sum_{1\leq i <j\leq k} \lk (L_i, L_j).\]
\end{thm}

As consequences of Theorem \ref{teo:obstruction_pb} we obtain two results; the first result is a formula for the unlinking number of positive links with positive-braid components (Proposition \ref{prop: unlinking pos and braid pos}).~%
The second result is the following combinatorial criterion to test whether a positive link is also positive-braid link.~%

\begin{prop}\label{prop:criterion from p to be bp}
Let $D$ be a positive diagram for $L =K_1 \cup ... \cup K_\ell$. Denote by $D_1,...,D_\ell$ the diagrams for $K_1,...,K_\ell$ obtained by deleting all but the corresponding component from $D$. If $L$ is  a positive-braid link, then
\begin{equation}
o(D) = \sum_i o(D_i)
\label{eq:pbcirc}
\end{equation}
where  $o(D')$ denotes the number of Seifert circles of $D'$.
\end{prop}

A classical obstruction to being a positive-braid link is fibredness;
a link is called \dfn{fibred} if its complement fibres over $\bS^1$, and positive-braid links are fibred~\cite{Stallings}.
The obstruction to being a positive-braid link given by Proposition \ref{prop:criterion from p to be bp} and by being fibred are independent. 
First, consider the $2$-component link $L6a1\{1\}$; this link is positive, has positive-braid components, but it is not fibred. This can be seen from the fact that the reduced Seifert graph of $D$ is not a tree  \cite[Theorem 5.11]{FKP}. 
However, the positive diagram $D$ for $L6n1\{1\}$ given in \cite{LinkInfo} satisfies \eqref{eq:pbcirc} -- cf.~Table \ref{tab:small_linksI}. In the other direction, we have the following example.

\begin{exa}\label{ex:pnotbp}
In Figure \ref{fig:p-non-bp} is depicted a positive diagram $D$ of a link $L = K_1 \cup K_2$.~%
A positive link is fibred if, and only if, the reduced Seifert graph of any positive diagram is a tree~\cite[Theorem~5.11]{FKP}. Using this criterion one can check directly  that $L$ is fibred. 
Clearly the components of $L$, being an unknot and a positive trefoil knot, are positive-braid knots.
One can verify that  $o(D)=6\neq 4 = o(D_1) + o(D_2)$, where $D_1$ and $D_2$ are as in Proposition \ref{prop:criterion from p to be bp}. %
It follows that $L$ cannot be a positive-braid link. In particular, $L$ is a fibred positive link, with positive-braid (actually algebraic) components which cannot be algebraic -- ie.~the link of an isolated singularity of a complex algebraic curve, which are cables of certain torus links, and known to be positive-braid links \cite{Milnor}. See also \cite{Dimca} for an overview on algebraic links.
\end{exa}

\subsection{Outline of the paper}

In Section \ref{sec:background} we review some background material. In Section \ref{sec:qp} we prove Theorems \ref{teo:MainFullQP} and \ref{teo:obstruction_qp}. Finally, in Section \ref{sec:pbandp} we concern ourselves with positive and positive-braid links.
Appendix A is dedicated to the proof of the fact that all links arise as sub-links of a quasi-positive link. Appendix B contains the table of positivity and braid-positivity of all prime links with less than $8$ crossings.
\subsection*{acknowledgements}
The author wishes to thank Alberto Cavallo for the useful discussions. Furthermore, the author wishes to extend his thanks to an anonymous referee, whose comments and suggestions led to a substantial improvement of the paper.

\section{Background material and notation}\label{sec:background}

The main scope of this section is to fix the notation and review some results needed in the proofs of the main theorems.~%
The section is organised as follows; we start with some facts concerning braids.~%
Then, we pass to the definition of self-linking number, and the description of some of its properties.~%
Afterward we  present some basic results on slice-torus link invariants.~%
We conclude with some basic facts on unlinking and splitting numbers.

\subsection{Braids}

The \dfn{braid group} on $n$-strands, here denoted by $B_n$, is the group generated by $\sigma_1, ..., \sigma_{n-1}$ (called \dfn{Artin generators}), and subject to the relations
\[ \sigma_{i}\sigma_{j} = \sigma_{j}\sigma_{i},\ \text{if }\vert i - j \vert > 1,\quad\text{and}\quad \sigma_{i}\sigma_{i+1}\sigma_i = \sigma_{i+1} \sigma_i \sigma_{i+1},\ \forall i,j.\]
Alternatively, one can define the braid group diagrammatically, the reader can consult \cite[Chapter 10]{Cromwell} for details, and refer to Figure~\ref{fig:braid} for a schematic description. 
A \dfn{braid} is a word in the Artin generators and their inverses.~%
In the rest of the paper no difference shall be made between the algebraic and the diagrammatic interpretations of the braid group.

By joining the start and end points of each strand of a braid,  as illustrated on the right side of Figure~\ref{fig:braid}, one obtains a link digram.~%
The link represented by such a diagram is  called the \dfn{closure} of the given braid.~%
A classical theorem of Alexander states that all links can be obtained as closed braids -- see \cite{Alexander}.

\begin{defn}
A braid is \dfn{quasi-positive} if it is the product of conjugates of Artin generators, and it is \dfn{positive} if it is the product of Artin generators.
A link is called \dfn{quasi-positive} (resp. \dfn{positive-braid}) if it is the closure of a quasi-positive (resp. positive) braid.
\end{defn}

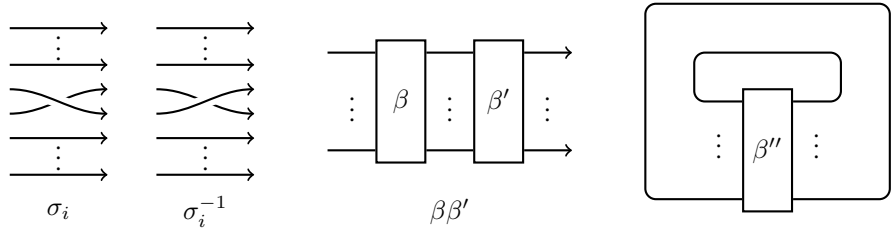
\begin{figure}[ht]
\begin{tikzpicture}[thick, scale =.65]
\draw[->] (-1,3.25) -- (1,3.25);
\node at (0,3) {$\vdots$};
\draw[->] (-1,2.5) -- (1,2.5);
\draw[->] (-1,1.5) .. controls +(.75,0) and +(-.75,0) .. (1,2);
\draw[white, line width = 5] (-1,2) .. controls +(.75,0) and +(-.75,0) .. (1,1.5);
\draw[->] (-1,2) .. controls +(.75,0) and +(-.75,0) .. (1,1.5);
\draw[->] (-1,1) -- (1,1);
\node at (0,.75) {$\vdots$};
\draw[->] (-1,.25) -- (1,.25);
\node at (0,-.5) {$\sigma_i$};

\begin{scope}[shift ={+(3,0)}]
\draw[->] (-1,3.25) -- (1,3.25);
\node at (0,3) {$\vdots$};
\draw[->] (-1,2.5) -- (1,2.5);
\draw[->] (-1,2) .. controls +(.75,0) and +(-.75,0) .. (1,1.5);
\draw[white, line width = 5] (-1,1.5) .. controls +(.75,0) and +(-.75,0) .. (1,2);
\draw[->] (-1,1.5) .. controls +(.75,0) and +(-.75,0) .. (1,2);
\draw[->] (-1,1) -- (1,1);
\node at (0,.75) {$\vdots$};
\draw[->] (-1,.25) -- (1,.25);
\node at (0,-.5) {$\sigma_i^{-1}$};
\end{scope}

\begin{scope}[shift  ={+(6.5,0)}]
\draw[->] (-1,2.75) -- (4,2.75);
\node at (-.5,1.75) {$\vdots$};
\draw[->] (-1,.75) -- (4,.75);
\draw[fill, white] (0,3) rectangle (1,0.5);
\draw (0,3) rectangle (1,0.5);
\node at (.5,1.75) {$\beta$};
\draw[fill, white] (2,3) rectangle (3,0.5);
\draw (2,3) rectangle (3,0.5);
\node at (2.5,1.75) {$\beta'$};
\node at (3.5,1.75) {$\vdots$};
\node at (1.5,1.75) {$\vdots$};

\node at (1.5,-.5) {$\beta\beta'$};
\end{scope}

\begin{scope}[shift ={+(13,-1)}]
\draw[rounded corners]  (0,2.75) rectangle (3,3.75);
\draw[rounded corners] (-1,.75) rectangle (4,4.75);
\draw[fill, white] (1,3) rectangle (2,0.5);
\draw (1,3) rectangle (2,0.5);
\node at (1.5,1.75) {$\beta''$};
\node at (2.5,2) {$\vdots$};
\node at (.5,2) {$\vdots$};

\end{scope}

\end{tikzpicture}
\caption{In the order from left to right: the Artin generator $\sigma_i$ and its inverse $\sigma_i^{-1}$, the product of two braids $\beta$ and $\beta'$, and the closure of a braid $\beta''$.}\label{fig:braid}
\end{figure}

We call a braid \dfn{alternating} if moving along each strand overpasses and underpasses alternate.~%
A braid is called \dfn{non-split} if its closure is not a split link. 

\begin{lemma}\label{lem:non-splitaltbraids}
A braid is non-split and alternating if, and only if, it satisfies following properties
\begin{enumerate}[(i)]
\item all $\sigma_i$ appear (possibly with negative exponents),
\item each $\sigma_i$ appears with exponents of a fixed sign (that is all positive or all negative),
\item $\sigma_{i}$ and $\sigma_{i+1}$ always appear with exponents of opposite signs.
\end{enumerate}
\end{lemma}
\begin{proof}
It is a straightforward verification that a braid satisfying (ii) and (iii) is alternating, and that it is not alternating otherwise. It is also immediate that if a braid does not satisfy (i) then it is split, since there are no crossing between a pair of adjacent strands. The only non-trivial point is that if a braid is alternating and satisfies (i) then it is non-split -- which follows from \cite[Corollary 7.6.4]{Cromwell}. 
\end{proof}

Let $\beta$ be a braid representing a link $L$ -- that is, the closure of $\beta$ is $L$ -- and  let $L'$ be a sub-link $L$.~%
By deleting from $\beta$ all strands which belong to $L\setminus L'$ in $\beta$'s closure, one obtains a braid $\beta'$ whose closure is $L'$.~%
We shall refer to $\beta'$ as the \dfn{sub-braid of $\beta$ associated to $L'$}.~%
Note that $\beta'$ does not belong to the same braid group as $\beta$ -- cf.~Figure \ref{fig:subbraid}.

\begin{figure}[ht]
\centering
\begin{tikzpicture}[thick, xscale = .45 , yscale = .75]
\begin{scope}[shift ={+(-2,0)}]
\draw[ ] (-1,2.5) -- (1,2.5);
\draw[ ] (-1,1) -- (1,1);
\draw[ ,red] (-1,1.5) .. controls +(.75,0) and +(-.75,0) .. (1,2);
\draw[white, line width = 5] (-1,2) .. controls +(.75,0) and +(-.75,0) .. (1,1.5);
\draw[ ,red] (-1,2) .. controls +(.75,0) and +(-.75,0) .. (1,1.5);
\end{scope}
\draw[ ] (-1,2.5) -- (1,2.5);
\draw[ ] (-1,1) -- (1,1);
\draw[ ,red] (-1,1.5) .. controls +(.75,0) and +(-.75,0) .. (1,2);
\draw[white, line width = 5] (-1,2) .. controls +(.75,0) and +(-.75,0) .. (1,1.5);
\draw[ ,red] (-1,2) .. controls +(.75,0) and +(-.75,0) .. (1,1.5);
\begin{scope}[shift ={+(2,0)}]
\draw[ ] (-1,2.5) -- (1,2.5);
\draw[ ] (-1,1) -- (1,1);
\draw[ ,red] (-1,1.5) .. controls +(.75,0) and +(-.75,0) .. (1,2);
\draw[white, line width = 5] (-1,2) .. controls +(.75,0) and +(-.75,0) .. (1,1.5);
\draw[ ,red] (-1,2) .. controls +(.75,0) and +(-.75,0) .. (1,1.5);
\end{scope}
\begin{scope}[shift ={+(4,.5)}]
\draw[ ] (-1,.5) -- (1,.5);
\draw[ ,red] (-1,1) -- (1,1);
\draw[ ,red] (-1,1.5) .. controls +(.75,0) and +(-.75,0) .. (1,2);
\draw[white, line width = 5] (-1,2) .. controls +(.75,0) and +(-.75,0) .. (1,1.5);
\draw[ ] (-1,2) .. controls +(.75,0) and +(-.75,0) .. (1,1.5);
\end{scope}
\begin{scope}[shift ={+(10,.5)}]
\draw[ ,red] (-1,.5) -- (1,.5);
\draw[ ] (-1,1) -- (1,1);
\draw[ ] (-1,1.5) .. controls +(.75,0) and +(-.75,0) .. (1,2);
\draw[white, line width = 5] (-1,2) .. controls +(.75,0) and +(-.75,0) .. (1,1.5);
\draw[ ,red] (-1,2) .. controls +(.75,0) and +(-.75,0) .. (1,1.5);
\end{scope}
\begin{scope}[shift ={+(6,-.5)}]
\draw[ ] (-1,2.5) -- (1,2.5);
\draw[ ,red] (-1,3) -- (1,3);
\draw[ ] (-1,1.5) .. controls +(.75,0) and +(-.75,0) .. (1,2);
\draw[white, line width = 5] (-1,2) .. controls +(.75,0) and +(-.75,0) .. (1,1.5);
\draw[ ,red] (-1,2) .. controls +(.75,0) and +(-.75,0) .. (1,1.5);
\end{scope}
\begin{scope}[shift ={+(12,-.5)}]
\draw[->,red] (-1,2.5) -- (1,2.5);
\draw[->] (-1,3) -- (1,3);
\draw[->,red] (-1,1.5) .. controls +(.75,0) and +(-.75,0) .. (1,2);
\draw[white, line width = 5] (-1,2) .. controls +(.75,0) and +(-.75,0) .. (1,1.5);
\draw[->] (-1,2) .. controls +(.75,0) and +(-.75,0) .. (1,1.5);
\end{scope}
\begin{scope}[shift ={+(8,0)}]
\draw[ ,red] (-1,2.5) -- (1,2.5);
\draw[ ,red] (-1,1) -- (1,1);
\draw[ ] (-1,1.5) .. controls +(.75,0) and +(-.75,0) .. (1,2);
\draw[white, line width = 5] (-1,2) .. controls +(.75,0) and +(-.75,0) .. (1,1.5);
\draw[ ] (-1,2) .. controls +(.75,0) and +(-.75,0) .. (1,1.5);
\end{scope}
\node at (5,.5) {$\beta$};

\begin{scope}[shift ={+(-7,-2)},scale = .6]
\begin{scope}[shift ={+(-2,0)}]
\draw[ ,red] (-1,1.5) .. controls +(.75,0) and +(-.75,0) .. (1,2);
\draw[white, line width = 5] (-1,2) .. controls +(.75,0) and +(-.75,0) .. (1,1.5);
\draw[ ,red] (-1,2) .. controls +(.75,0) and +(-.75,0) .. (1,1.5);
\end{scope}
\draw[ ,red] (-1,1.5) .. controls +(.75,0) and +(-.75,0) .. (1,2);
\draw[white, line width = 5] (-1,2) .. controls +(.75,0) and +(-.75,0) .. (1,1.5);
\draw[ ,red] (-1,2) .. controls +(.75,0) and +(-.75,0) .. (1,1.5);
\begin{scope}[shift ={+(2,0)}]
\draw[ ,red] (-1,1.5) .. controls +(.75,0) and +(-.75,0) .. (1,2);
\draw[white, line width = 5] (-1,2) .. controls +(.75,0) and +(-.75,0) .. (1,1.5);
\draw[ ,red] (-1,2) .. controls +(.75,0) and +(-.75,0) .. (1,1.5);
\end{scope}
\begin{scope}[shift ={+(4,.5)}]
\draw[ ,red] (-1,1) -- (1,1);
\draw[ ,red] (-1,1.5) .. controls +(.75,0) and +(-.75,0) .. (1,2);
\end{scope}
\begin{scope}[shift ={+(10,.5)}]
\draw[ ,red] (-1,.5) -- (1,.5);
\draw[ ,red] (-1,2) .. controls +(.75,0) and +(-.75,0) .. (1,1.5);
\end{scope}
\begin{scope}[shift ={+(6,-.5)}]
\draw[ ,red] (-1,3) -- (1,3);
\draw[ ,red] (-1,2) .. controls +(.75,0) and +(-.75,0) .. (1,1.5);
\end{scope}
\begin{scope}[shift ={+(12,-.5)}]
\draw[->,red] (-1,2.5) -- (1,2.5);
\draw[->,red] (-1,1.5) .. controls +(.75,0) and +(-.75,0) .. (1,2);
\end{scope}
\begin{scope}[shift ={+(8,0)}]
\draw[ ,red] (-1,2.5) -- (1,2.5);
\draw[ ,red] (-1,1) -- (1,1);
\end{scope}
\node at (5,.5) {$\beta'$};
\end{scope}

\begin{scope}[shift ={+(9,-2)},scale = .6]
\begin{scope}[shift ={+(-2,0)}]
\draw[ ] (-1,2.5) -- (1,2.5);
\draw[ ] (-1,1) -- (1,1);
\end{scope}
\draw[ ] (-1,2.5) -- (1,2.5);
\draw[ ] (-1,1) -- (1,1);
\begin{scope}[shift ={+(2,0)}]
\draw[ ] (-1,2.5) -- (1,2.5);
\draw[ ] (-1,1) -- (1,1);
\end{scope}
\begin{scope}[shift ={+(4,.5)}]
\draw[ ] (-1,.5) -- (1,.5);
\draw[ ] (-1,2) .. controls +(.75,0) and +(-.75,0) .. (1,1.5);
\end{scope}
\begin{scope}[shift ={+(10,.5)}]
\draw[ ] (-1,1) -- (1,1);
\draw[ ] (-1,1.5) .. controls +(.75,0) and +(-.75,0) .. (1,2);
\end{scope}
\begin{scope}[shift ={+(6,-.5)}]
\draw[ ] (-1,2.5) -- (1,2.5);
\draw[ ] (-1,1.5) .. controls +(.75,0) and +(-.75,0) .. (1,2);
\end{scope}
\begin{scope}[shift ={+(12,-.5)}]
\draw[->] (-1,3) -- (1,3);
\draw[->] (-1,2) .. controls +(.75,0) and +(-.75,0) .. (1,1.5);
\end{scope}
\begin{scope}[shift ={+(8,0)}]
\draw[ ] (-1,1.5) .. controls +(.75,0) and +(-.75,0) .. (1,2);
\draw[white, line width = 5] (-1,2) .. controls +(.75,0) and +(-.75,0) .. (1,1.5);
\draw[ ] (-1,2) .. controls +(.75,0) and +(-.75,0) .. (1,1.5);
\end{scope}
\node at (5,.5) {$\beta''$};
\end{scope}
\end{tikzpicture}
\caption{On the top, the braid $\beta = \sigma_2^3 \sigma_1\sigma_3\sigma_2\sigma_1\sigma_3\in B_4$. On the bottom, the sub-braids $\beta' = \sigma_1^3\in B_2$ (left) and $\beta''=\sigma_1\in B_2$ (right) associated to the red and black components, respectively.}\label{fig:subbraid}
\end{figure}
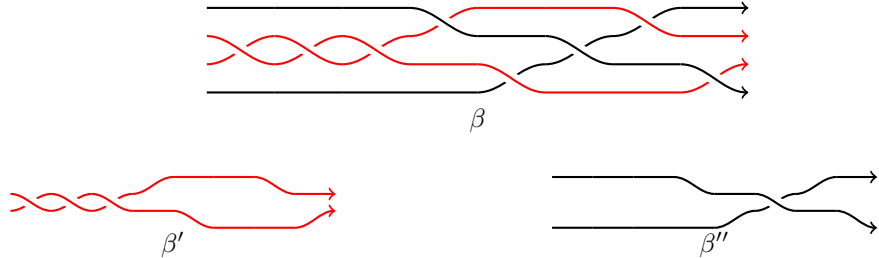

\subsection{The self-linking number}
To each $n$-stranded braid $\beta$ is associated an integer called \dfn{self-linking number}, which is defined as
\[ sl(\beta) = w(\beta) - n\]
where $w(\beta)$ denotes the \dfn{writhe}, or algebraic crossing count, of $\beta$.~%
Rudolph proved in \cite{RudolphObstructionToSlice} that if $\beta$ is a quasi-positive braid, and if $L$ denotes the closure of $\beta$, then
\begin{equation}
sl(\beta) = - \chi_4(L).
\label{eq:sl=chi}
\end{equation}
This fact will be repeatedly used throughout the paper, and it is central in our proofs.
Another key fact, which will be extensively used in our proofs, is presented in the following lemma.

\begin{lemma}\label{lem:KeyLemma}
Let $L$ be a link partitioned into disjoint sub-links $L_1$,...,$L_k$. Given a braid $\beta$ representing $L$, denote by $\beta_i$ the sub-braid associated to $L_i$. Then, we have
\begin{equation} sl(\beta) - \sum_{i=1}^{k} sl(\beta_i) = 2 \sum_{1\leq i <j \leq k} \lk (L_i,L_j)
\label{eq:keylemma}\end{equation}
\end{lemma}
\begin{proof}
Since the number of strands of the $\beta_i$ adds up to the number of strands in $\beta$, the left-hand side of Equation \eqref{eq:keylemma} is the signed count of all crossings in $\beta$ which do not belong to any $\beta_i$. Which is twice the total linking number between the $L_j$s, as claimed.
\end{proof}

\subsection{Slice-torus link invariants}\label{subs:slicetorus}
Recall from the introduction that a \dfn{slice-torus link invariant} is a special $\bR$-valued link concordance invariant which vanishes on unlinks.~%
We do not need the precise definition of slice-torus link invariant, but only some formal properties, thus we refer the reader to \cite{CavalloCollari} for it.~%

Let $D$ be an oriented link diagram.~%
The \dfn{oriented resolution} of $D$ is the set of oriented circles obtained by replacing each crossing in $D$ with its oriented smoothing -- see Figure \ref{fig:orientedres}.~%
The circles appearing in the oriented resolution are called \dfn{Seifert circles} and their number is denoted by $o(D)$.~%
Denote by $s_+(D)$ (resp. $s_-(D)$) the number of connected components of the graph obtained by smoothing all negative (resp.~positive) crossings, and flattening all positive (resp.~negative) crossings.~%
Denote by $w(D)$ the algebraic crossing count (or \dfn{writhe}) of $D$.~%
Finally, denote by $\ell_s(D)$ the number of connected components of the graph obtained by flattening all crossings in $D$, regardless of the sign -- in particular, it follows that $\ell_s(D)\leq s_\pm(D)$.
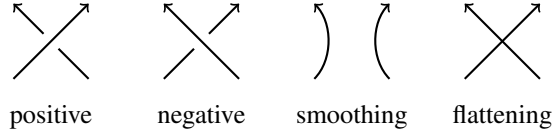
\begin{figure}[H]
\centering
\begin{tikzpicture}[thick, scale = .5]
\draw[->](-3,-1) -- (-5,1);
\draw[white, line width =8](-5,-1) -- (-3,1);
\draw[->](-5,-1) -- (-3,1);
\node at (-4,-2) {positive};

\draw[->](-1,-1) -- (1,1);
\draw[white, line width =8](1,-1) -- (-1,1);
\draw[->](1,-1) -- (-1,1);
\node at (0,-2) {negative};

\draw[->](3,-1) .. controls +(.5,.5) and  +(.5,-.5)  .. (3,1);
\draw[->](5,-1) .. controls +(-.5,.5) and  +(-.5,-.5)  ..  (5,1);
\node at (4,-2) {smoothing};

\draw[->](7,-1) -- (9,1);
\draw[->](9,-1) -- (7,1);
\node at (8,-2) {flattening};
\end{tikzpicture}
\caption{A positive crossing, a negative crossing, the corresponding oriented smoothing, and the corresponding flattening.}\label{fig:orientedres}
\end{figure}
The above-defined (diagram-dependent) combinatorial quantities can be used to estimate the value of all slice-torus link invariants.~%
More precisely, we have the following combinatorial bounds.
\begin{prop}[{\cite[Theorem 1.3]{CavalloCollari}}]\label{pro:combinatorialbound}
Let $\nu$ be a slice-torus link invariant.~For each oriented diagram $D$ representing the $\ell$-component link $L$, we have
\[\frac{w(D) - o(D) +2s_+(D) +\ell - 2 \ell_s(D)}{2} \leq \nu(L) \leq \frac{w(D) + o(D) - 2s_-(D) -\ell + 2 \ell_s(D)}{2}\]
\end{prop}

An immediate consequence is the following refinement of the slice-Bennequin inequality.~%

\begin{cor}\label{cor:slice-torusBennequin}
For each braid $\beta$ representing an $\ell$-component link $L$,  and any slice-torus link invariant $\nu$, we have
\[ sl(\beta) \leq  2\nu (L) -\ell.\]
\end{cor}
\begin{proof}
If we use the closure of $\beta\in B_n$ as a diagram, then the number of strands $n$ equals the number of Seifert circles.~%
The inequality follows from Proposition \ref{pro:combinatorialbound}, and from $\ell_s \leq s_+$.
\end{proof}

\begin{rem}\label{rem:homogeneous}
The upper and lower bounds in Proposition \ref{pro:combinatorialbound} coincide if, and only if, $D$ is an \dfn{homogeneous diagram}  \cite{Cromwell}.~%
This follows from \cite[Theorem 3.4]{Abe} -- which is stated for knots, but holds for non-split links -- together with \cite[Corollary 7.6.4]{CromwellBook}, and \cite[Definition 2.1(B)]{CavalloCollari}.~%
Examples of homogeneous diagrams are positive diagrams and alternating diagrams  \cite[Section 7.6]{CromwellBook}.
\end{rem}
\subsection{Unlinking and splitting numbers}
In this subsection we collected two basic facts concerning unlinking and splitting numbers.~%
Let us by recalling some basic terminology. If a crossing change features a single component it is called \dfn{self-crossing change},  otherwise it is called \dfn{mixed crossing change}.
A \dfn{splitting sequence} (resp. \dfn{unlinking sequence}) for a link $L$ is a finite sequence of links $L= L^{(0)}, L' ,..., L^{(k)}$ such that: $L^{(i)}$ is obtained from $L^{(i+1)}$ via a single crossing change, and $L^{(k)}$ is the split union of knots (resp. an unlink). A splitting sequence featuring only mixed crossing changes is called \dfn{strong splitting sequence}. The unlinking, splitting, and strong splitting numbers are the minimal lengths of the eponymous sequences.

\begin{lemma}[{\cite[Lemma $2.1$]{CavalloCollariConway}}]\label{lem:lkleqm}
Let $L$ be a link.~%
If there exists a splitting sequence with $m$ mixed crossing changes, then
\[ \vert \lk\vert (L) \leq m,\]
where $\vert \lk \vert$ denotes the \dfn{absolute linking number}, that is the sum of the absolute value of the linking numbers between the components of $L$ -- not to be confused with $\vert \lk (L) \vert$.
\end{lemma}

If the absolute linking number equals the strong splitting number, one can say something about the unlinking number, and splitting number $\wsp$.~%

\begin{lemma}\label{lem:sharplinkingbound}
Let $L = K_1 \cup \cdots \cup K_\ell$ be a link.~%
If  $  \ssp (L) =\vert \lk \vert (L)$, then 
\begin{enumerate}[(i)]
\item $\wsp (L) = \vert \lk \vert (L)$;
\item $u(L)- \sum_{i=1}^{\ell} u(K_i) = \vert \lk \vert (L) $.
\end{enumerate}
\end{lemma}
\begin{proof}
The equality in (i) has been proved in \cite[Remark 2.3]{CavalloCollariConway}.~%
Consider a minimal unlinking sequence  for $L$ with $m$ mixed crossing changes and $s$ self crossing changes.~%
Since an unlinking sequence is also a splitting sequence, Lemma \ref{lem:lkleqm} implies $\vert \lk \vert (L)  \leq m$.~%
Note that mixed crossing change do not change the isotopy class of the components.~%
Thus, the sum of the unknotting numbers of the $K_i$ is less than, or equal to, $s$.~%
It follows that
\[ \vert \lk \vert (L) + \sum_{i} u(K_i) \leq m + s= u(L).\]
The opposite inequality, and hence the equality, follows from a simple observation; a particular unlinking sequence is given by a minimal strong splitting sequence followed by a minimal unknotting sequence for each component.~%
\end{proof}

\section{Concordance to quasi-positive links}\label{sec:qp}

In this section we prove our obstructions to the existence of a concordance between a given link and quasi-positive links (Theorem \ref{teo:MainFullQP}), and an obstruction to quasi-positivity (Theorem \ref{teo:obstruction_qp}).~%
We conclude the section with some remarks on the concordance self-linking number $sl_{c}$ -- i.e.~the maximal self-linking number in a concordance class.~%
More concretely, we compute $sl_c$ for the closure of pure braids, and obtain a necessary condition for a pure braid to be concordant to a quasi-positive link.~%
Finally, we identitify all alternating pure braids whose closure is concordant to a quasi-positive link.

\subsection{Obstruction to quasi-positivity and concordance to quasi-positive links}
This short subsection is dedicated to the proofs of Theorems \ref{teo:MainFullQP} and \ref{teo:obstruction_qp}.~%

\begin{proof}[of Theorem \ref{teo:MainFullQP}]
Let $L=L_1 \cup \cdots \cup L_k$ be a quasi-positive link, and let $\beta$ be a quasi-positive braid representing $L$.~%
Denote by $\beta_i \in B_{n_i}$ the sub-braid of $\beta$ associated to $L_i$, and by $\ell_i$ the number of components of $L_i$. %
From \cite[Theorem 3.2]{CavalloCollari}, it follows that
\[ 2\nu(L) - \ell = sl(\beta). \]
Therefore, we have
\[ 2 \left( \nu(L) - \sum_{ i =1}^{k} \nu(L_i)\right) = sl(\beta) - \sum_{i=1}^{k} (2\nu(L_i) - \ell_i) \leq sl(\beta) - \sum_{i=1}^{k} sl(\beta_i),\]
where the last inequality follows from Corollary \ref{cor:slice-torusBennequin}.~%
The statement now follows from Lemma \ref{lem:KeyLemma}, and from the concordance invariance of both $\nu$ and the linking matrix.
\end{proof}

We already gave an application of Theorem \ref{teo:MainFullQP}, thus we go straight to the proof of Theorem \ref{teo:obstruction_qp}.

\begin{proof}[of Theorem \ref{teo:obstruction_qp}]
Let $L=L_1 \cup \cdots \cup L_k$ be a quasi-positive link partitioned into sub-links.~%
Consider a quasi-positive braid $\beta$ representing $L$, and denote by $\beta_i$ the sub-braid representing~$L_i$.~%
Fix a slice-torus link invariant $\nu$, and assume that 
\[ \nu(L) - \sum_{i=1}^{k} \nu(L_i) = \sum_{i<j} \lk (L_i,L_j).\]
Since $\beta$ is quasi-positive, it follows from \eqref{eq:OldConditionsQP}, \eqref{eq:sl=chi}, and Lemma \ref{lem:KeyLemma} that
\[ \sum_{i=1}^{k} (2 \nu(L_i) - \ell_i) = \sum_{i=1}^{k} sl(\beta_i). \]
Since, by Corollary \ref{cor:slice-torusBennequin}, we have $sl(\beta_i) \leq (2 \nu(L_i) - \ell_i)$ the equality follows for each $i$.
\end{proof}

We conclude the section with an example.
In the following example we make use of the connected sum.  The connected sum of $L$ and $L'$ is denoted by $L \#_{K,K'} L'$, where $K$ and $K'$ are the component used to perform the sum \cite[Section 4.6]{Cromwell}. Whenever $K$ or $K'$, possibly both, is immaterial or clear from the context, it shall be omitted.

\begin{exa}\label{exa:obs-qp}
Let $L = L_1 \cup ... \cup L_k$ be a quasi-positive link satisfying Equation \eqref{eq:hp_obs_qp} with $L_1$ an unknot, eg.~the positive $(2,2n)$ torus link $T(2,2n)$. Let $K$ be a slice knot with $sl_{max}(K) < -1$, eg.~$6_1$. It can be seen that the link $L' = L\#_{L_1} K = K \cup L_2 \cup ... \cup L_k$ is concordant to $L$, and thus satisfies both Equation \eqref{eq:hp_obs_qp} and $2\nu(L') -\ell = -\chi_4(L')$. However, by hypothesis
\[ sl_{max}(K)  <  -1 = 2 \nu(K) -1. \]
Hence, by Theorem \ref{teo:obstruction_qp}, $L'$ cannot be quasi-positive.
\end{exa}

\subsection{Concordance self-linking number}

Recall from the introduction, that the \dfn{maximal self-linking number}
\[ sl_{c}(L) = \max \left\lbrace sl(\beta) \vert\ L\text{ is the closure of }\beta \right\rbrace,\]
and the \dfn{concordance self-linking number} 
\[ sl_{c}(L) = \max \left\lbrace sl(\beta) \vert\ \text{the closure of  }\beta\text{ is concordant to }L \right\rbrace,\]
are well-defined integer-valued link invariant, and concordance invariant, respectively.~%
Furthermore, it follows from \eqref{eq:OldConditionsQP} and \eqref{eq:sl=chi} that if an $\ell$-component link $L$ is concordant to a quasi-positive link, then
\begin{equation}
sl_{c}(L) = 2\nu(L) - \ell = -\chi_4(L),
\label{eq:slc=nu}
\end{equation}
for each slice-torus link invariant $\nu$.~%
Hence, the concordance self-linking number can be used to obstruct the concordance to quasi-positive links.~%
However, the concordance self-linking number is extremely difficult to compute, nonetheless we are able to compute it for pure braid closures.

\begin{prop}\label{prop:slmetc vs lk}
Let $L$ be a link partitioned into sub-links $L_1$,...,$L_k$. Then, we have
\[ sl_{\ast}(L) \leq  \sum_{i=1}^{k} sl_{\ast}(L_i) + 2 \sum_{1\leq i <j \leq k} \lk (L_i,L_j),\quad \text{where }\ast\in\{ {max}, {c} \}.\]
\end{prop}
\begin{proof}
First we prove the case $\ast = max$. Let $\beta$ be a braid representing $L$ and realising the maximal self-linking number. Denote by $\beta_i$ the sub-braid of $\beta$ associated to $L_i$. By Lemma \ref{lem:KeyLemma}, we have
\[ sl_{max}(L) = sl(\beta) =\sum_{i=1}^{k} sl(\beta_i) + 2 \sum_{1\leq i <j \leq k} \lk (L_i,L_j) \leq \sum_{i=1}^{k} sl_{max}(L_i) + 2 \sum_{1\leq i <j \leq k} \lk (L_i,L_j).\]
Now, we can prove the case $*=c$. %
Since $sl_{max}(L_i) \leq sl_{c}(L_i)$, we have
\begin{equation}
\label{eq:slmaxandslc}
sl_{max}(L) \leq \sum_{i=1}^{k} sl_{c}(L_i) + 2 \sum_{1\leq i <j \leq k} \lk (L_i,L_j).
\end{equation}
Notice that any strong concordance establishes a bijection between the components of two links, and thus a bijection between sub-links.~%
Therefore, for any link $L'$ concordant to $L$, we can find a partition into sub-links $L'_1,...,L'_k\subset  L'$ such that $L_i$ is concordant to $L_i'$, and $\lk(L'_i,L'_j) =\lk(L_i,L_j) $ -- since the  linking matrix is a concordance invariant.~%
Hence, applying Equation \eqref{eq:slmaxandslc} we get
\[ sl_{max}(L') \leq \sum_{i=1}^{k} sl_{c}(L'_i) + 2 \sum_{1\leq i <j \leq k} \lk (L'_i,L'_j) = \sum_{i=1}^{k} sl_{c}(L_i) + 2 \sum_{1\leq i <j \leq k} \lk (L_i,L_j).\]
The statement follows by taking $L'$ concordant to $L$ and such that~$sl_{max}(L') = sl_c (L)$.
\end{proof}

\begin{exa}\label{ex:slcandqp}
Consider the link $L(k)$ in Figure \ref{fig:L_hk}.~%
By Proposition \ref{prop:slmetc vs lk} we have $sl_{c}(L(k))\leq -2 $. Moreover, by Theorem \ref{teo:ExampleConctoQP} we have $-\chi_{4}(L(k)) = 2k -2$. Since it does not satisfy Equation \eqref{eq:slc=nu}, the link $L(k)$ cannot be concordant to a quasi-positive link for each $k\geq 1$.
\end{exa}

\begin{cor}\label{cor:slc and unknotted comp}
If $L$ has unknotted components, then $sl_{c}(L) \leq  2\lk (L) - \ell$.
\end{cor}

\begin{prop}\label{prop:slc and pure}
If $L = K_1 \cup ... \cup K_\ell$ is concordant to the closure of a pure braid $\beta$, then
\[ sl_c(L) = sl(\beta) = 2 \lk(L) - \ell. \]
In particular, if $L$ is also concordant to a quasi-positive link, then $ \nu(L) = \lk(L)$.
\end{prop}
\begin{proof}
A pure braid $\beta$ has as many strands as components in its closure. ~%
Furthermore, in $\beta$ there are only crossings between strands belonging to distinct components. In particular, we have that the components of the closure of $\beta$ are unknots, and that $2\lk(L) = w(\beta)$.~%
It follows immediately that $2\lk (L) - \ell = sl(\beta) \leq sl_c(L)$.~%
The equality follows from Corollary~\ref{cor:slc and unknotted comp}.~%
Finally, if $L$ is concordant to a quasi-positive link, then $2 \nu(L) - \ell = sl_c(L)$ by \eqref{eq:slc=nu}, and the result follows.
\end{proof}

As an application of Proposition \ref{prop:slc and pure} we characterise all the closures of alternating pure braids which are concordant to a quasi-positive link.
\begin{prop}\label{prop:altpurebraids}
Let $L$ be the closure of an alternating pure braid. Then, $L$ is concordant to a quasi-positive link if, and only if, $L$ is the split union  of (possibly trivial) positive $(2,n)$-torus links.
\end{prop}
\begin{proof}
A split union of positive $(2,n)$-torus links is the closure of an alternating positive pure braid.

Let $\beta$ be a $n$-stranded alternating pure braid representing $L$, and denote by $D$ the diagram obtained as the closure of $\beta$.
Note that $D$ is alternating.
Denote by $D_1$, ..., $D_{k}$ the connected components of $D$, thus $\ell_s(D) =k$. 
Since there are no crossings between  the $D_i$'s, $L$  is the split union of the links represented by them.
To conclude, we prove that each $D_i$ represents either a positive $(2,n)$-torus link or an unknot.

The combinatorial bound in Proposition \ref{pro:combinatorialbound} is sharp for alternating diagrams -- Remark \ref{rem:homogeneous}.
Adding the fact that $\beta$ is a pure braid, and thus $n= o(D) = \ell$, we obtain
\[  2 \nu(L) - \ell = sl(\beta) + 2s_+(D) - 2 \ell_s= sl(\beta) + \sum_{i=1}^{k}2 s_+(D_i) - 2 k, \]
where the last equality is due to the fact that the $D_i$'s are the connected components of $D$.
From Proposition \ref{prop:slc and pure}, and from \eqref{eq:slc=nu} (recall that, by hypothesis, $L$ is concordant to a quasi-positive link), it follows  that $sl(\beta) = 2\nu(L) - \ell$.
Comparing the two expressions for $2\nu(L) - \ell$, we obtain
\[ \sum_{i=1}^{k} s_+(D_i) =  k. \]
Since $s_+(D_i) \geq 1$ by definition, it follows that $s_+(D_i) =1$.

Every connected component of an alternating diagram is an alternating diagram.
From this fact, and from \cite[Corollary~7.6.4]{CromwellBook}, it follows that each connected component of an alternating diagram is an alternating diagram representing a non-split link.
By definition, each $D_i$ is the closure of a $n_i$-stranded sub-braid of $\beta$, denoted by $\beta_i$, which is non-split and alternating for what we just said.
Now, Lemma~\ref{lem:non-splitaltbraids} implies that $s_+(D_i) \in \{ \lfloor n_i/2\rfloor , \lceil n_i/2\rceil \}$. Since we proved that  $s_+(D_i) =1$, we obtain $n_i \in \{ 1, 2 \}$. \\
It follows that each $\beta_i$ represents either an unknot or a non-trivial $(2,n)$-torus link. The latter satisfies Equation \eqref{eq:OldConditionsQP} if and only if it is positive, giving the desired result. 
\end{proof}

\begin{cor}
The Borromean link $B$ is not concordant to any quasi-positive link.
\end{cor}
\begin{proof}
Since $B$ is the closure of $(\sigma_1\sigma_2^{-1})^3\in B_3$, the result follows from Proposition \ref{prop:altpurebraids}.
\end{proof}
\section{Positive and positive-braid links}\label{sec:pbandp}
In this final section we prove the results concerning positive and positive-braid links.~%
\subsection{Positive links}
We call a link \dfn{simply-linked} if it admits a diagram such that the number crossings between distinct components equals twice the absolute linking number.~%
Clearly, positive links are simply-linked.~%
More generally, a link is simply-linked if, and only if, it has a diagram such that the crossings between each pair of fixed components have the same sign.
The splitting numbers of simply-linked links are equal to the absolute linking number, more precisely we have the following.

\begin{prop}\label{prop:sharp lk bound I}
Let $L$ be simply-linked.
Then, we have the following equalities
\[ \ssp(L) = \wsp(L) = \vert \lk\vert(L) = u(L) - \sum_{i=1}^{\ell} u(K_i)\]
\end{prop}
\begin{proof}
By Lemma \ref{lem:lkleqm} and Lemma \ref{lem:sharplinkingbound}, it is sufficient to prove the inequality $\ssp (L)\leq  \vert \lk \vert (L)$.
It is an easy exercise to prove that for any link $\ssp$ is at most half the number of crossings between different components in any diagram.
If $L$ is simply-linked, we can take a diagram with precisely $2|\lk|(L)$ crossings between different components. The result follows.
\end{proof}

\begin{proof}[of Theorem \ref{teo:obstruction_p}]
For all positive links we have $\vert \lk\vert (L) = \lk(L)$.~%
Since positive links are simply-linked, the result follows directly from Proposition \ref{prop:sharp lk bound I}.
\end{proof}

Now, we compare the obstruction given  by Theorem \ref{teo:obstruction_p} with other existing obstructions to positivity.\\
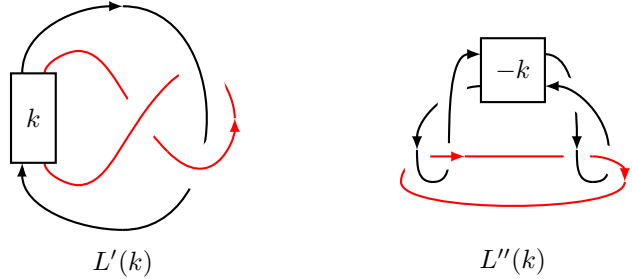
\begin{figure}[h]
\begin{tikzpicture}[thick, scale =.85]
\draw[latex-] (-1,0) .. controls +(0,.5) and +(-.5,0) .. (0,1);
\draw[white, line width =10] (-.5,0) .. controls +(0,.5) and +(-.5,0) .. (0,1.5);
\draw[-latex] (-.5,0) .. controls +(0,.5) and +(-.5,0) .. (0,1.5);

\draw[latex-](1.5,0) .. controls +(0,.5) and +(.5,0) .. (1,1.5);
\draw[white, line width =10] (2,0) .. controls +(0,.5) and +(.5,0) .. (1,1);
\draw[-latex] (2,0) .. controls +(0,.5) and +(.5,0) .. (1,1);

\draw[red] (-1.25,-.5) .. controls +(0,.25) and +(-.25,0) .. (-.75,-.1);
\draw[white, line width =10]  (-.75,-.5) .. controls +(-.25,0) and +(0,-.25) .. (-1,0);
\draw[] (-.75,-.5) .. controls +(-.25,0) and +(0,-.25) .. (-1,0);

\draw (-.75,-.5) .. controls +(.25,0) and +(0,-.25) .. (-.5,0);
\draw[white, line width =10] (-.25,-.1) -- (-.75,-.1); 
\draw[red, latex-] (-.25,-.1) -- (-.75,-.1); 
\node at (.5,-1.75) {$L''(k)$};

\begin{scope}[shift ={+(2.5,0)}]

\draw[red] (-1.25,-.1) -- (-.75,-.1);
\draw[white, line width =10]  (-.75,-.5) .. controls +(-.25,0) and +(0,-.25) .. (-1,0);
\draw[] (-.75,-.5) .. controls +(-.25,0) and +(0,-.25) .. (-1,0);

\draw (-.75,-.5) .. controls +(.25,0) and +(0,-.25) .. (-.5,-.1);
\draw[white, line width =10] (-.25,-.5) .. controls +(0,.25) and +(.25,0) .. (-.75,-.1); 
\draw[red, latex-] (-.25,-.5) .. controls +(0,.25) and +(.25,0) .. (-.75,-.1); 
\end{scope}

\draw[red] (-.25,-.1) -- (1.25,-.1);

\draw[red] (-1.25,-.5) .. controls +(0,-.5) and +(0,-.5) .. (2.25,-.5);

\draw (0, .75) rectangle (1,1.75);
\node at (.5,1.25) {$-k$};

\begin{scope}[shift ={+(-7,.5)}, scale = .35]

\draw[-latex] (-.5,2) .. controls +(0,2) and +(-2,0) .. (4,5);
\draw[latex-] (-.5,-2) .. controls +(0,-2) and +(-2,0) .. (4,-5);

\draw[red] (.5,2)  .. controls +(0,.5) and +(-.5,0) .. (2,3) .. controls +(1.5,0) and +(-2,2) .. (6,-1.5);
\pgfsetlinewidth{20*\pgflinewidth}
\draw[white] (.5,-2) .. controls +(0,-.5) and +(-.5,0) .. (2,-3) .. controls +(1.5,0) and +(-2,-2) .. (6,1.5);
\pgfsetlinewidth{0.05*\pgflinewidth}
\draw[red] (.5,-2) .. controls +(0,-.5) and +(-.5,0) .. (2,-3) .. controls +(1.5,0) and +(-2,-2) .. (6,1.5);

\draw[red] (9,0)  .. controls +(0,1) and +(2,2) .. (6,1.5);
\pgfsetlinewidth{20*\pgflinewidth}
\draw[white] (4,-5) .. controls +(5,0) and +(5,0) .. (4,5);
\pgfsetlinewidth{0.05*\pgflinewidth}
\draw[] (4,-5) .. controls +(5,0) and +(5,0) .. (4,5);

\pgfsetlinewidth{20*\pgflinewidth}
\draw[white] (9,0) .. controls +(0,-1) and +(2,-2) .. (6,-1.5);
\pgfsetlinewidth{0.05*\pgflinewidth}
\draw[latex-, red] (9,0) .. controls +(0,-1) and +(2,-2) .. (6,-1.5);

\draw (-1,2) rectangle (1,-2);
\node at (0,0) {$k$};
\node at (4,-6.5) {$L'(k)$};
\end{scope}
\end{tikzpicture}
\caption{The links $L'(k)$ and $L''(k)$. The boxes indicate $k\geq 0$ full twist with the same sign as their label.}\label{fig:lk=2}
\end{figure}
\begin{proof}[Proposition \ref{pro:indep_pos_obs}]
Consider the link $L'(k)$ on the left side of Figure \ref{fig:lk=2}. A quick inspection of the diagram shows that $\lk(L'(k))= k-1$. Moreover, since the diagram in Figure \ref{fig:lk=2} is alternating we can compute the signature combinatorially \cite{CombSign}, and get $\sigma(L'(k)) = 1-2k$. It follows that $L'(k)$ has positive linking number, positive components, and negative signature for $k\geq 2$. However, $L'(k)$ is not a positive link; in fact, $\wsp(L'(k)) \neq \ssp(L'(k))$ by \cite[Proposition 3.7]{CavalloCollari}, in contrast with Theorem \ref{teo:obstruction_p}. \\
Now, consider the link $L'''(n) = T(2,2) \# T(2,2n+1)^*$.  The link $L'''(n)$ has one negative component isotopic to $T(2,2n+1)^*$, $\lk (L'''(n)) =1$, and it is also simply linked. Therefore, Equation \eqref{eq:obs_pos} holds by Proposition \ref{prop:sharp lk bound I}. What is left is to compute the signature of $L'''(n)$. The signature is additive with respect to connected sum~\cite[Theorem 6.7.5]{Cromwell}. Moreover, we have $\sigma(T(2,2)) = -1$, while $\sigma(T(2,2n+1)^*)= 2n$. Thus, we obtain $\sigma(L'''(n)) =2n -1$, which is positive for $n\geq 1$.
\end{proof}

Now we can characterise all positive links with unlinking number $1$, starting from those with $\lk=1$ and two trivial components. 

\begin{prop}\label{prop:hopf}
If $L $ is a positive link with $\lk(L)=1$ and two unknotted components, then $L$ is the positive Hopf link.
\end{prop}
\begin{proof}
Let $D$ be a positive diagram representing $L$.~%
Since $\lk(L)=1$ there are only two crossings in $D$ which involve both components.~%
It follows that $D$ must be, up to planar isotopy, as illustrated in Figure \ref{fig:hopf}, and thus $L$ is a connected sum of a Hopf link, a knot $K_1$, and a knot $K_2$.~%
Since the components of $L$ are unknotted, it follows that $K_1$ and $K_2$ must be trivial.
\end{proof}

\begin{figure}[h]
\begin{tikzpicture}[thick, yscale =.75]
\draw[-latex] (1.75,.5) arc (0:180:1.25);
\draw[line width =10, white] (.75,.5) arc (0:180:1.25);
\draw[-latex] (.75,.5) arc (0:180:1.25);

\draw[latex-]  (.75,-.5) arc (0:-180:1.25);
\draw[line width =10, white] (1.75,-.5) arc (0:-180:1.25);
\draw[latex-](1.75,-.5) arc (0:-180:1.25);

\draw (-2,-.5) rectangle (-1.5,.5);

\draw (-1,-.5) rectangle (-.5,.5);
\draw (2,-.5) rectangle (1.5,.5);

\draw (1,-.5) rectangle (.5,.5);

\node at (1.75,0) {$T_1'$};
\node at (-1.75,0) {$T_0$};
\node at (.75,0) {$T_0'$};
\node at (-.75,0) {$T_1$};

\draw[->] (2.5,0) -- (3.5,0);
\node at (3,.5) {isotopy};

\begin{scope}[shift ={+(6,0)}]
\draw[] (1.75,.5) arc (0:180:1.25);
\draw[line width =10, white] (.75,.5) arc (0:180:1.25);
\draw[-latex] (.75,.5) arc (0:180:1.25);

\draw[]  (.75,-.5) arc (0:-180:1.25);
\draw[line width =10, white] (1.75,-.5) arc (0:-180:1.25);
\draw[latex-](1.75,-.5) arc (0:-180:1.25);

\draw (.75,-.5) -- (.75,.5);
\draw (-.75,-.5) -- (-.75,.5);

\draw (-2,-.5) rectangle (-1.5,.5);

\draw (2,-.5) rectangle (1.5,.5);

\node at (1.75,0) {$K_2$};
\node at (-1.75,0) {$K_1$};
\end{scope}
\end{tikzpicture}
\caption{The depiction of two positive diagrams of a link with two components, and linking number one. Each box indicates a (possibly trivial) $(1,1)$-tangle.}
\label{fig:hopf}
\end{figure}
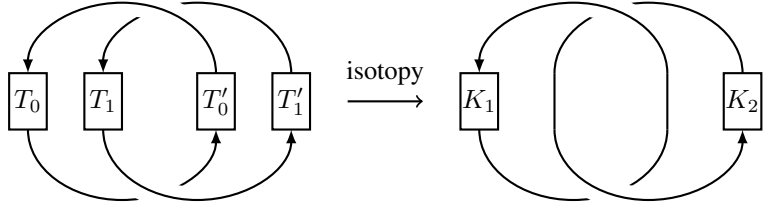

\begin{proof}[of Proposition \ref{prop:u=1}]
Let $D$ be a positive diagram representing a link $L = K_1 \cup ... \cup K_\ell$, and assume $L$ to have unlinking number one.~%
By Theorem \ref{teo:obstruction_p}, we have that
\begin{equation}
\lk(L) + \sum_{i} u(K_i) = u(L) =1.
\label{eq:tmpu=1}
\end{equation} 
Since the unknotting number is non-negative, one obtains that $\lk(L) \leq 1$.~%
Positive links have non-negative linking numbers, therefore $\lk(L)\in \{0,1\}$.
If $\lk(L)=0$, then by Theorem \ref{teo:obstruction_p} we have $\ssp(L)=0$ and  $\sum_{i} u(K_i) = u(L) = 1$.~%
It follows that $L$ is the split union of knots, and precisely one among the $K_i$ has unknotting number one.~%
Therefore, $L$ is the split union of an unlink and a positive knot whose unknotting number is one -- which is a twist knot by Nakamura's result \cite[Theorem 5.1]{Nakamura}.~%

If $\lk(L)=1$, then \eqref{eq:tmpu=1} implies that all components have unknotting number $0$, and therefore are unknotted.~%
Moreover, since positive diagrams are simply-linked, $D$ has precisely two crossing between distinct components.~%
It follows that $L$ is the split union of an unlink, and a positive link with two unknotted components and linking number $1$.~%
Hence, by Proposition \ref{prop:hopf}, $L$ is the split union of an unlink and a positive Hopf link.
\end{proof}

The proof of Proposition \ref{prop:u=2} proceeds similarly to the proof of Proposition \ref{prop:u=1}.

\begin{proof}[ of Proposition \ref{prop:u=2}]
Let $D$ be a positive diagram representing a link $L = K_1 \cup ... \cup K_\ell$, and assume $u(L)=2$.~%
From Theorem \ref{teo:obstruction_p}, it follows that
\begin{equation}
\lk(L) + \sum_{i} u(K_i) = u(L) = 2,
\label{eq:tmpu=2}
\end{equation} 
and thus $\lk(L)\in \{0, 1, 2\}$.~%

If $\lk(L) =2$, then all components must have unknotting number $0$, and hence be unknots (case a.).

If $\lk(L) = 1$, then $D$ must be the union of some knot diagrams and a diagram of the form shown in Figure \ref{fig:hopf} -- by the same reasoning as in the proof of Proposition~\ref{prop:hopf}.~%
Only one among the components of $L$ has unknotting number one; thus, by Propositions \ref{prop:u=1} and~{\cite[Theorem 5.1]{Nakamura}}, $L$ is either the split union of an unlink and a connected sum of a Hopf link with a twist knot (case b.), or the split  union of an unlink, a twist knot, and a Hopf link (case c.).~%

Finally, if $\lk(L)=0$ then $L$ is the split union of knots by Theorem \ref{teo:obstruction_p}.~%
In this case, the sum of the unknotting numbers of the components of $L$ is $2$.~%
Hence, either there exists $K_i$ and $K_j$ with $i\neq j$ which are twist knots (case d.), or there is a single $K_i$ with unknotting number two (case e.).

What is left to show is that there are infinitely many knot in each case.~%
It is known that there are infinitely-many twist knots -- eg.~\cite[Chapter 9, E 9.6]{BZH}, hence it is sufficient to show that the following families are infinite;
\begin{enumerate}
\item  positive links with linking number two and trivial components; 
\item  positive knots with unknotting number $2$.
\end{enumerate}
On the right of Figure \ref{fig:lk=2} is depicted the link $L''(k)$. Note that $\lk(L''(k)) = 2$, and the components  of $L''(k)$ are unknotted.~%
It is known that two alternating, non-split, and reduced diagrams of the same link must have the same number of crossings \cite{KauffmanTait,MurasugiTait,ThistlethwaiteTait}. Since the diagrams in Figure \ref{fig:lk=2} are alternating, non-split, and reduced, we have $L''(k) \neq L''(h)$ for $h\neq k$. Hence family (i) is infinite.

Now, we shall argue that $u(K\# K')=2$ when $K$ and $K'$ are positive twist knots. 
Then, family (ii) must be infinite since there are infinitely-many twist knots.~%
The unlinking number is sub-additive, thus $u(K\# K') \leq u(K) +u(K') =2$.~%
The sum of two positive knots is a positive knot.~%
Hence $u(K\#K')=1$ if, and only if, $K\#K'$ is a twist knot by \cite[Theorem 5.1]{Nakamura}.~%
However, twist knots are prime, and therefore it must be $u(K\#K') \geq 2$.
\end{proof}

\subsection{Positive-braid links}
In this final subsection we prove Theorem \ref{teo:obstruction_pb}, and its consequences.

\begin{proof}[of Theorem \ref{teo:obstruction_pb}]
Note that the  first equation in the statement follows directly from Theorem~\ref{teo:obstruction_p}, and the second equation. Therefore, it is sufficient to prove the second equation.\\
Let $\beta$ be a positive braid for $L = L_1 \cup ... \cup L_k$, and let $\beta_i$ the sub-braid of $\beta$ associated to $L_i$. %
Clearly, both $\beta$ and the $\beta_i$'s are positive (hence quasi-positive) braids. Comparing \eqref{eq:OldConditionsQP} and \eqref{eq:sl=chi}, we get $2\nu(L) - \ell= sl(\beta)$ and $2\nu(L_i) - \ell_i = sl(\beta_i)$. By Lemma \ref{lem:KeyLemma}, we have
\begin{equation}
 \nu(L) -\sum_{i=1}^{k} \nu(L_i)= \sum_{1\leq i <j\leq k} \lk (L_i, L_j).
 \label{eq:pbI}
 \end{equation}
Comparing \eqref{eq:pbI} with Theorem \ref{teo:obstruction_p} we have
\begin{equation}
\nu(L) -\sum_{i=1}^{\ell} \nu(K_i)= u(L) -\sum_{i=1}^{\ell} u(K_i)
 \label{eq:pbII}
 \end{equation}
where $ L =K_1 \cup \dots \cup K_\ell$.
Now, $L_i =  K_{r_{1,i}} \cup \dots \cup K_{r_{\ell_i,i}}$ is a positive-braid link  for each $i$, hence
\[ u(L) - \sum_{i=1}^{k} u(L_i) =  u(L) - \sum_{i=1}^{k} u(K_i) - \sum_{i=1}^{k} \left( u(L_i)- \sum_{j=1}^{\ell_i} u(K_{r_{j,i}}) \right) = \nu(L) -\sum_{i=1}^{k} \nu(L_i).\]
In the last equality we applied \eqref{eq:pbII} to both $L$ and each $L_i$. 
\end{proof}

\begin{figure}[h]
\includegraphics[scale=.19]{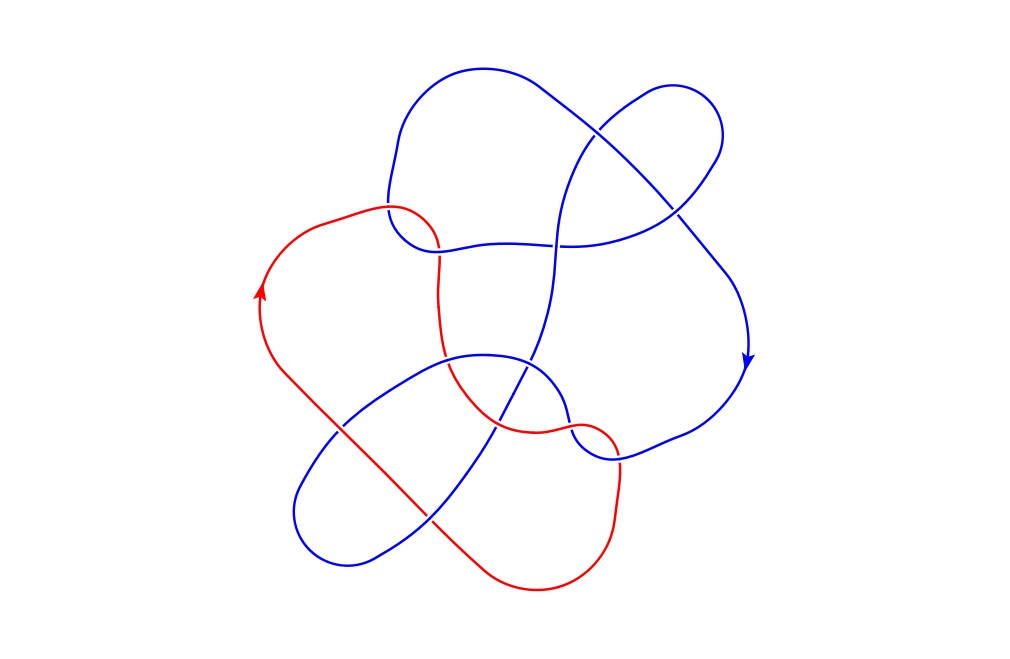}
\caption{A fibred positive link, which is not a positive-braid link, whose components are positive-braid knots.}
\label{fig:p-non-bp}
\end{figure}

\begin{exa}\label{ex:pnotbp2}
In Figure \ref{fig:p-non-bp} is shown a positive diagram of a link $L = K_1 \cup K_2$, where $K_1$ (red) is an unknot, and $K_2$ (blue) a trefoil knot.~%
Hence, $u(K_1)=0$ and $u(K_2) = 1$ \cite{KnotInfo}. 
Since $\lk(L) =4$, by Theorem~\ref{teo:obstruction_p}, we get $u(L)=5$.~%
However, by Proposition~\ref{pro:combinatorialbound}, we have that $\nu(L)=4$, $\nu(K_1) = 0$, and $\nu(K_2) =1$.~%
This contrasts with Theorem~\ref{teo:obstruction_pb}, hence $L$ is not a positive-braid link.~%
\end{exa}
Kawamura computed the unlinking number of positive-braid links; that is, he proved the following.
\begin{prop}[{\cite[Theorem 1.3]{Kawamura02}}]\label{prop:obstruction_pb_2}
Let $L$ be a positive-braid link, and $D$ a diagram obtained from closing a positive braid representing $L$. Then,
\[ 2u(L) = x(D) -o(D) + \ell,\]
where $x$ denotes the number of crossings.
\end{prop}

Comparing Proposition \ref{prop:obstruction_pb_2} with \cite[Theorem 1.3]{CavalloCollari} (cf.~Proposition \ref{pro:combinatorialbound}, note that for positive diagrams $s_+ = \ell_s$) one obtains that
\begin{equation}
 u(L) = \nu(L)
\label{eq:u=nu}
\end{equation}
for each positive-braid link $L$, and each slice-torus link invariant $\nu$.~%
Note that \eqref{eq:u=nu} can be used to give an alternative proof of Theorem \ref{teo:obstruction_pb};
Equation \eqref{eq:u=nu} implies the second assertion in Theorem \ref{teo:obstruction_pb}.~%
In turn, this implies the whole theorem via Theorem \ref{teo:obstruction_p}.

\begin{rem}\label{rem:ubpfromposdiag}
Since $u= \nu$ for positive-braid links, it follows from \cite[Theorem 1.3]{CavalloCollari} that the unlinking number of a positive-braid link can be computed using the formula in Proposition \ref{prop:obstruction_pb_2} from any positive diagram, and not only from the closures positive braids.
\end{rem}%
\begin{rem}
A quick check on Knotinfo \cite{KnotInfo} shows that $16$ positive prime knots with less than $12$ crossings do not satisfy Equation \eqref{eq:u=nu} with $\nu = \tau$, and thus are not positive-braid knots.
\end{rem}

\begin{proof}[of Proposition \ref{prop:criterion from p to be bp}]
Let $D$ be a positive diagram of a positive-braid link $L =K_1 \cup ... \cup K_\ell$. %
Denote by $D_i$ the diagram obtained from $D$ by deleting all but the strands belonging to $K_i$.~%
Note that $D$ and the $D_i$'s are positive diagrams, and $L$ and the $K_i$'s are positive-braid links. Hence, Remark~\ref{rem:ubpfromposdiag} implies  that $2u(D) = x(D) - o(D) +\ell$ and $2u(K_i) = x(D_i) -o(D_i) + 1$.~%
Moreover, by Theorem~\ref{teo:obstruction_p}, we have
\[ u(L) - \sum_i u(K_i) = \lk(L).\]
Finally, a simple observation, akin to the proof of Lemma \ref{lem:KeyLemma}, shows that in this case
\[ x(D) -\sum_i x(D_i) = w(D) -\sum_i w(D_i) =2 \lk(L).\]
Thus, putting all the above facts together, we obtain 
\[ x(D) -\sum_i x(D_i) = 2 \lk(L) = 2(u(L) - \sum_i u(K_i)) =  x(D) -o(D) - \sum_{i} (x(D_i) - o(D_i)). \]
It follows that $o(D) - \sum_i o(D_i) = 0$, which concludes the proof.~%
\end{proof}

Finally, combining Kawamura's result, ie.~Proposition \ref{prop:obstruction_pb_2}, and Theorem \ref{teo:obstruction_p} we obtain a formula for the unlinking number for positive links with positive-braid components.

\begin{prop}\label{prop: unlinking pos and braid pos}
Let $L =K_1 \cup ... \cup K_\ell$ be a positive link, and $D$ a positive diagram for $L$. Denote by $D_i$ the diagram obtained by deleting all the components but $K_i$ from $D$. If $K_1,...,K_\ell$ are positive braid knots, then
\[ u(L) =  \frac{x(D) - \sum_{i} o(D_i) + \ell}{2},\]
where $x$ denotes the number of crossings.
\end{prop}
\begin{proof}
Since $L$ is positive, it follows from Theorem \ref{teo:obstruction_p} that
\[ 2u(L) = 2\sum_i u(K_i) +2 \lk(L) = \sum_i 2u(K_i) + (x(D) - \sum_i x(D_i)),\]
where the second equality follows from the definition of linking number and the positivity of the diagram -- cf.~proof of Proposition \ref{prop:criterion from p to be bp}. \\
By Proposition \ref{prop:obstruction_pb_2} we have that  $2u(K_i) = x(D_i) -o(D_i) + 1$ for each $i$.~%
Therefore 
\[ 2u(L) =  \sum_i (x(D_i) -o(D_i) + 1) + (x(D) - \sum_i x(D_i))  = x(D) - \sum_i o(D_i) +\ell.\]
\end{proof}

\begin{exa}
Consider $12$-crossing diagram $D$ representing the link $L = K_1 \cup K_2$ depicted in Figure~\ref{fig:p-non-bp}. %
Denote by $D_i$ the diagram associated to $K_i$, for $i=1,2$. One can check that $o(D_1) + o(D_2) =4$.~%
Therefore, we obtain $u(L) = (x(D) -o(D_1) - o(D_2) + \ell)/2 = 5$ -- cf.~Example \ref{ex:pnotbp2}.
\end{exa}

\appendix

\section{Any link is a sub-link of a quasi-positive link}\label{app:sublinks}

In this short appendix we prove that any link arises as sub-link of a quasi-positive link.~%
\begin{thm}\label{teo:allsublinks}
Any link $L$ is a sub-link of a quasi-positive link $L'$. Furthermore, $L'$ can be obtained from $L$ by adding a single unknotted component.
\end{thm}
\begin{proof}
Let  $\beta = \sigma_{i_1}^{\epsilon_1} \cdot \cdots \cdot  \sigma_{i_k}^{\epsilon_k} \in B_n$ be a braid representative for $L$, where $\epsilon_1,...,\epsilon_k \in\{\pm 1\}$.
Our aim is to define a new braid $\beta' \in B_{n+1}$ which is quasi-positive, and whose closure contains $L$ as a sub-link.
First, define the following braid
\[ w_r =\begin{cases}  \sigma_{n+1} \cdot ... \cdot \sigma_{i_r}\cdot \sigma_{i_r} \cdot  ... \cdot \sigma_{n+1} & \text{if } \epsilon_r =  -1\text{}\\ 1 & \text{otherwise}\end{cases},\quad\text{for each $r \in \{ 1,...,k\}$ }\] 
Notice that the permutation associated to each $w_r$ is trivial. 
Now, define
\[ \beta' = (\sigma_{i_1}^{\epsilon_1}w_1) \cdot \cdots \cdot  (\sigma_{i_k}^{\epsilon_k}w_k)\in B_{n+1}. \]
Clearly, the permutation associated to $\beta'$ fixes $n+1$, and behaves as the permutation associated to $\beta$ on $\{1,...,n\}$. This fact implies that the closure of $\beta'$ has an unknotted component $U$ which starts and ends on the $(n+1)$-th strand.
The braid $\beta'$ is quasi-positive since it is a product of quasi-positive braids. Indeed, we have
\[\sigma_{i_r}^{\epsilon_{i_r}}w_r =\begin{cases}  \sigma_{i_r}^{-1} (\sigma_{n+1}  ...  \sigma_{i_r+1})\sigma_{i_r} (\sigma_{i_r}   ...  \sigma_{n+1}) & \text{if } \epsilon_r =  -1\text{,}\\ \sigma_{i_r} & \text{otherwise.}\end{cases}\]
Finally, we have to show that the link $L$ is a sub-link of the closure $L'$ of $\beta'$. We remark that the unknotted component $U$ has no self-crossings in $\beta'$, and the $w_r$ are precisely the crossings between $U$ and that the rest of the components.~%
Therefore, the braid associated to $L'\setminus U$ is obtained removing all the $w_r$ from $\beta'$ thus re-obtaining $\beta$. 
\end{proof}

\section{Table of positive, and positive-braid, prime links with less than $8$  crossings}\label{app:table}

\definecolor{darkgreen}{RGB}{0, 102, 51}
\newcommand{\cmark}{\textcolor{darkgreen}{\ding{51}}}%
\newcommand{\xmark}{\textcolor{red}{\ding{55}}}%

In this appendix we compile a table of positive and positive-braid prime links with less than $8$ crossings. 
The naming convention in Table \ref{tab:small_linksI} follows LinkInfo \cite{LinkInfo}. In particular, we do not distinguish between $L$, $L^*$, $rL$, and $rL^*$.
A link $L$ is marked as positive (resp. as positive-braid) if either $L$ or $L^*$ is positive (resp. a positive-braid).
Positive links (resp. positive-braid links) are indicated with a tick mark (\cmark) under the column `P' (resp. `BP'). Similarly, non-positive and non-braid-positive links will be indicated with a cross mark (\xmark) in the corresponding column. For each positive (or positive-braid) link, the diagram (resp. braid) in \cite{LinkInfo} is positive  (up to mirror).
In the remaining columns, labelled by \ref{teo:obstruction_p} and \ref{teo:obstruction_pb}, we indicate whether Theorem \ref{teo:obstruction_p} and Theorem \ref{teo:obstruction_pb}, respectively, provide the desired obstruction (\cmark) or not (\xmark).
When Theorems \ref{teo:obstruction_p} and \ref{teo:obstruction_pb} do not work, the obstruction used shall be denoted by one of the following symbols
\begin{itemize}
\item[$\sigma$] sign of the signature, and non-negativity of the linking matrix;
\item[c] positivity of the components, and non-negativity of the linking matrix;
\item[f] the link is not fibred -- recall that a positive link is fibred if, and only if, the reduced Seifert graph of a positive diagram is a tree \cite[Theorem 5.1]{FKP}. In particular, positive-braid links are fibred~\cite{Stallings}.
\end{itemize}
We warn the reader that the signature and linking matrix in \cite{LinkInfo} might be coherent only up to mirror.
\begin{table}[h]\vspace*{-3ex}
\begin{tabular}{clccccc|}
Name &  P &  \ref{teo:obstruction_p}   &  BP &  \ref{teo:obstruction_pb}\\
\hline\hline
L2a1\{0\}  & \cmark &     & \cmark &      \\
L2a1\{1\} & \cmark &     &  \cmark &      \\
L4a1\{0\} &  \cmark &      & \xmark &   \cmark   \\
L4a1\{1\} & \cmark  &     &  \cmark &     \\
L5a1\{0\} & \xmark & \cmark   & \xmark &      \\
L5a1\{1\} & \xmark & \cmark    &  \xmark &     \\
L6a1\{0\} &  \xmark$^\sigma$ & \xmark  & \xmark &     \\
L6a1\{1\} & \cmark  &      &  \xmark$^{\rm f}$  &   \xmark  \\
L6a2\{0\} & \cmark &       & \xmark &  \cmark     \\
L6a2\{1\} & \cmark &       & \xmark &  \cmark    \\
L6a3\{0\} & \cmark &    & \cmark &      \\
L6a3\{1\} & \cmark &     &  \xmark &  \cmark   \\
L6a4\{0,0\} & \xmark & \cmark   &  \xmark &     \\
L6a4\{1,0\} & \xmark & \cmark  & \xmark &      \\
L6a4\{0,1\} & \xmark & \cmark  &  \xmark &     \\
L6a4\{1,1\} & \xmark & \cmark    &  \xmark &     \\
L6a5\{0,0\} & \cmark &    &  \xmark$^{\rm f}$ & \xmark \\
L6a5\{1,0\} & \xmark & \cmark    & \xmark &      \\
L6a5\{0,1\} & \xmark & \cmark   & \xmark &      \\
L6a5\{1,1\} & \xmark & \cmark    &  \xmark &     \\
L6n1\{0,0\} & \xmark & \cmark    & \xmark &      \\
L6n1\{1,0\} & \xmark & \cmark   &  \xmark &     \\
\end{tabular}%
\begin{tabular}{clcccc}
Name &  P &  \ref{teo:obstruction_p}   &  BP &   \ref{teo:obstruction_pb}\\
\hline\hline
L6n1\{0,1\} & \cmark &   & \cmark &      \\
L6n1\{1,1\} & \xmark &  \cmark  &  \xmark &     \\
L7a1\{0\} & \xmark  &  \cmark    & \xmark &      \\
L7a1\{1\} & \xmark  &  \cmark    & \xmark &      \\
L7a2\{0\} & \cmark &       & \xmark & \cmark     \\
L7a2\{1\} & \xmark$^c$  & \xmark    &  \xmark &     \\
L7a3\{0\} & \xmark  &  \cmark    & \xmark &      \\
L7a3\{1\} & \xmark  &  \cmark    & \xmark &      \\
L7a4\{0\} & \xmark  &  \cmark    & \xmark &      \\
L7a4\{1\} & \xmark  &  \cmark    &  \xmark &     \\
L7a5\{0\} & \xmark$^\sigma$  &  \xmark   & \xmark &      \\
L7a5\{1\} & \xmark$^\sigma$ &  \xmark   &  \xmark &     \\ 
L7a6\{0\} & \xmark  &  \cmark    &  \xmark &     \\
L7a6\{1\} & \xmark  &  \cmark   & \xmark &      \\
L7a7\{0,0\} & \xmark & \cmark   & \xmark &      \\
L7a7\{1,0\} & \xmark & \cmark    & \xmark &      \\
L7a7\{0,1\} & \xmark & \cmark   &  \xmark &     \\
L7a7\{1,1\} & \cmark &    & \xmark$^{\rm f}$  & \xmark     \\
L7n1\{0\} &  \cmark &      & \cmark &      \\
L7n1\{1\} & \xmark$^c$  &  \xmark  &  \xmark &    \\
L7n2\{0\} & \xmark & \cmark   & \xmark &      \\
L7n2\{1\} & \xmark & \cmark    &  \xmark &     \\
\end{tabular}
\caption{Positivity and braid-positivity of the $44$ prime links with less than $8$ crossings.}\label{tab:small_linksI}
\end{table}

\end{document}